\documentclass[reqno,10pt]{amsart}
\usepackage{amsmath,amsthm,amsfonts,amssymb,latexsym,epsfig}
\usepackage{hyperref,times}

\newtheorem{theorem}{Theorem}[section]
\newtheorem{lemma}{Lemma}[section]
\newtheorem{corollary}{Corollary}[section]
\theoremstyle{remark}
\newtheorem{remark}{Remark}
\numberwithin{equation}{section}

\newcommand{\lb}{\left}
\newcommand{\rb}{\right}
\begin{document}
\title[On some $P$-$Q$ mixed modular equations of degree 5]{On some $P$-$Q$ mixed modular equations of degree 5}
\author[M. S. Mahadeva Naika, S. Chandankumar and M. Harish]{M. S. Mahadeva Naika, S. Chandankumar and M. Harish }
\address[M. S. Mahadeva Naika]{Department of Mathematics, Bangalore University,  Central College
Campus, Bangalore-560 001, Karnataka, India}\email{msmnaika@rediffmail.com}
\address[S. Chandankumar]{Department of Mathematics, M. S. Ramaiah University of Applied Sciences, Peenya Campus, Peenya 4th Phase, Bengaluru-560 058, Karnataka, India}\email{{chandan.s17@gmail.com}}
\address[M. Harish]{Department of Mathematics, R. V. College of Engineering, Mysore Road, R. V. Vidyanikethan, Bengaluru, Karnataka 560 059, India}\email{{numharish@gmail.com}}
\begin{abstract}
In his second notebook, Ramanujan recorded total of 23 $P$-$Q$ modular equations involving theta-functions $\varphi(q)$, $\psi(q)$ and $f(-q)$. In this paper, modular equations analogous to those recorded by Ramanujan are obtained involving $f(-q)$. As a consequence, several values of quotients of theta-function are evaluated. \end{abstract}
\maketitle
\begin{center}
	{\bf Dedicated to Prof. C. Adiga on the occasion of his 62$^{nd}$ Birthday}
\end{center}
\section{Introduction}
For $|q|<1,$ let $(a;q)_\infty$
denote the infinite product $\displaystyle \prod_{n=0}^\infty(1-aq^{n})$, where $a$, $q $ are complex numbers and $f(a,b)$ be the Ramanujan theta-function:
\begin{equation*}
f(a,b):=\sum_{n=-\infty}^{\infty}a^{n(n+1)/2} b^{n(n-1)/2},\,\,\,|ab|<1,
\end{equation*}
The following definitions of theta-functions  $\varphi$, $\psi$ and $f$ follows as special cases of $f(a,b):$
\begin{align}
\varphi(q)&:=f(q,q)=\sum_{n=-\infty}^{\infty}q^{n^2},\\
 \psi(q)&:=f(q,q^3) =\sum_{n=0}^{\infty}q^{n(n+1)/2},\\
 f(-q)&:=f(-q,-q^2)=\sum_{n=-\infty}^{\infty}(-1)^n q^{n(3n-1)/2}.
\end{align}
The ordinary or Gaussian hypergeometric function is defined by
$$_2F_1(a,b;c;z):=\sum_{n=0}^{\infty}\frac{\lb(a\rb)_n \lb(b\rb)_n}{\lb(c\rb)_n n!}z^n,\ \ \ 0\leq|z|<1,$$
where $a$, $b$, $c$ are complex numbers, $c\neq0,-1,-2,\ldots$, and $$(a)_0=1,\ \ (a)_n=a(a+1)\cdots(a+n-1)\ \ \textrm{for any positive integer }n.$$
Let $K(k)$ be the complete elliptic integral of the  first kind of modulus $k$. Recall that
\begin{equation}\label{ee11}
K(k):=\int_0^{\frac{\pi}{2}}\frac{d\phi}{\sqrt{1-k^2\sin^2\phi}}
=\frac{\pi}{2}\sum_{n=0}^{\infty}\frac{\left(\frac{1}{2}\right)^2_n}{\left(n!\right)^2}k^{2n}
=\frac{\pi}{2}\varphi^2(q),\,\,\,\,\,(0<k<1),
\end{equation}
and set $K'=K(k')$, where $k'=\sqrt{1-k^2}$  is the so called complementary modulus of $k$. It is classical to set $q(k)=e^{-\pi K(k')/K(k)}$ so that $q$ is one-to-one increases from 0 to 1.

\noindent In the same manner introduce $L_1=K(\ell_1), {L'_1}=K(\ell'_1) $ and suppose that the following equality
\begin{equation}\label{ee12}
n_1\frac{K'}{K}=\frac{L'_1}{L_1}
\end{equation}
holds for some positive integer $n_1$. Then a modular equation of degree $n_1$ is a relation between the moduli $k$ and $\ell_1$ which is induced by \eqref{ee12}.  Following Ramanujan, set $\alpha=k^2$ and $\beta=\ell_1^2$. We say that $\beta$ is of degree $n_1$ over $\alpha$. The multiplier $m$, corresponding to the degree $n_{1}$, is defined by
\begin{equation}\label{a1}
m=\frac{K}{L_1}=\frac{\varphi^2(q)}{\varphi^2(q^{n_1})},
\end{equation}
for $q=e^{-\pi K(k')/K(k)}$.

\noindent Let $K$, $K'$, $L_1$, $L_1'$, $L_2$, $L_2'$, $L_3$ and $L_3'$ denote complete elliptic integrals of the first kind corresponding, in pairs, to the moduli $\sqrt{\alpha}$, $\sqrt{\beta}$, $\sqrt{\gamma}$ and $\sqrt{\delta}$, and their complementary moduli, respectively.  Let $n_1$, $n_2$ and $n_3$ be positive integers such that $n_3=n_1n_2$.  Suppose that the equalities
\begin{equation}\label{ee14}
n_1\frac{K'}{K}=\frac{L_1'}{L_1},\ \ n_2\frac{K'}{K}=\frac{L_2'}{L_2} \ \ \textrm {and}\ \ n_3\frac{K'}{K}=\frac{L_3'}{L_3},
\end{equation}
hold.  Then a ``mixed'' modular equation is a relation between the moduli $\sqrt{\alpha}$, $\sqrt{\beta}$, $\sqrt{\gamma}$ and $\sqrt{\delta}$ that is induced by \eqref{ee14}.  We say that $\beta$, $\gamma$ and $\delta$ are of degrees $n_1$, $n_2$ and $n_3$, respectively over $\alpha$.  The multipliers $m=K/L_1$ and $m'=L_2/L_3$ are algebraic relation involving $\alpha$, $\beta$, $\gamma$ and $\delta$.

 At scattered places of his second notebook \cite{SR2}, Ramanujan recorded a total of nine $P$--$Q$ ``mixed'' modular relations of degrees 1, 3, 5 and 15. These relations were proved by B. C. Berndt and L. -C. Zhang \cite{BCBLCZ1}, \cite{BCBLCZ2} and the same has been reproduced in the book by Berndt \cite[pp. 214-235]{BCB2}. In \cite{BAM1}, S. Bhargava, C. Adiga and M. S. Mahadeva Naika have established several new $P$--$Q$ ``mixed'' modular relations with four moduli. For more information one can see \cite{MSMCKSBH} and \cite{MSMCKSBH1}. Motivated by all these works, we establish some new modular equations of ``mixed'' degrees and as an application, we establish some new general formulas for the explicit evaluations of a remarkable product of theta function.

 In Section \ref{sec2}, we collect some identities which are useful in proofs of our main results. In Section \ref{sec3}, we establish several new modular equations of degree 5. In Section \ref{sec4}, we establish several new $P$--$Q$ ``mixed'' modular equations akin to those recorded by Ramanujan in his notebooks.

 Mahadeva Naika, M. C. Maheshkumar and Bairy \cite{MSMMCMKSB}, have defined a new remarkable product of theta-functions $b_{s,t}$:
\begin{equation}\label{bmn}
b_{s,\,\,t}= \frac{ t e^{ \frac{-(t-1)\pi}{4}
\sqrt{\frac{s}{t}}}\psi^2\left(-e^{-\pi\sqrt{st}}\right)
\varphi^2\left(-e^{-2\pi\sqrt{st}}\right)}{\psi^2\left(-e^{-\pi\sqrt{\frac{s}{t}}}\right)
\varphi^2\left(-e^{-2\pi\sqrt{\frac{s}{t}}}\right)},
\end{equation}
 where $s$, $t$ are real numbers such that $s>0$ and  $t\geq1.$ They have established some new general formulas for the explicit evaluations of $b_{s,t}$ and computed some particular values of $b_{s,t}$. In Section \ref{sec5}, we establish some new modular relations connecting a remarkable product of theta-functions $b_{s,5}$ with $b_{r^2s,5}$ for $r=$ 2, 4 and 6 and explicit values of $b_{s,5}$ are deduced.

\section{Preliminary results}\label{sec2}
In this section, we list some of the relevant identities which are useful in the proofs of our results.
\begin{lemma}\cite[Ch. 17, Entry 12 (i) and (iii), p. 124]{BCB1}
For\, $0<x<1$, let
\begin{eqnarray}
&&f(e^{-y})=\sqrt z2^{-1/6}\{x(1-x)e^y\}^{1/24},\label{b1}\\&&
f(e^{-2y})=\sqrt z2^{-1/3}\{x(1-x)e^y\}^{1/12},\label{b2}
\end{eqnarray}\end{lemma}
\textrm{where} $z:= \ _2F_1(\frac{1}{2},\frac{1}{2};1;x)$ and $\displaystyle y:=\pi\frac{_2F_1(\frac{1}{2},\frac{1}{2};1;1-x)}{_2F_1(\frac{1}{2},\frac{1}{2};1;x)}.$
\begin{lemma} \cite[Ch. 16, Entry 24 (ii) and (iv), p. 39]{BCB1} We have
\begin{eqnarray}
&&f^3(-q)=\varphi^2(-q)\psi(q),\label{b10}\\
&&f^3(-q^2)=\varphi(-q)\psi^2(q).\label{b11}
\end{eqnarray}
\end{lemma}
\begin{lemma} \cite[Ch. 19, Entry 13 (xii) and (vii), pp. 281-282]{BCB1}\\
If $\beta$ is of degree 5 over $\alpha$, then
\begin{eqnarray}
&&\lb(\frac{\beta}{\alpha}\rb)^{1/4}+\lb(\frac{1-\beta}{1-\alpha}\rb)^{1/4} -\lb(\frac{\beta\lb(1-\beta\rb)}{\alpha\lb(1-\alpha\rb)}\rb)^{1/4}=m,\label{b4}\\&&
\lb(\frac{\alpha}{\beta}\rb)^{1/4}+\lb(\frac{1-\alpha}{1-\beta}\rb)^{1/4} -\lb(\frac{\alpha\lb(1-\alpha\rb)}{\beta\lb(1-\beta\rb)}\rb)^{1/4}=\frac{5}{m},\label{b5}
\end{eqnarray}\end{lemma}
where $m$ is the multiplier for degree 5.
\begin{lemma}\cite[p. 55]{SR1}
 If $X:=\dfrac{f(-q)}{q^{1/6}f(-q^{5})}$ and $Y:=\dfrac{f(-q^2)}{q^{1/3}f(-q^{10})},$
then
\begin{equation}\label{b16}
XY+\frac{5}{XY}=\left(\frac{Y}{X}\right)^3+\left(\frac{X}{Y}\right)^3.
\end{equation}
\end{lemma}
\begin{lemma}\cite[p. 55]{SR1}
 If $X:=\dfrac{f(-q)}{q^{1/6}f(-q^{5})}$ and $Y:=\dfrac{f(-q^3)}{q^{1/2}f(-q^{15})},$
then
\begin{equation}\label{b7}
(XY)^3+\lb(\frac{5}{XY}\rb)^3+\lb[\lb(\frac{X}{Y}\rb)^6- \lb(\frac{Y}{X}\rb)^6\rb]+9\lb[\lb(\frac{X}{Y}\rb)^3+\lb(\frac{Y}{X}\rb)^3\rb]=0.
\end{equation}
\end{lemma}
\begin{lemma}\cite[p. 55]{SR1}
 If $X:=\dfrac{f(-q)}{q^{1/6}f(-q^{5})}$ and $Y:=\dfrac{f(-q^4)}{q^{2/3}f(-q^{20})},$
then
\begin{equation}\begin{split}\label{b8}
&(XY)^3+\lb(\frac{5}{XY}\rb)^3=\lb(\frac{X}{Y}\rb)^5+\lb(\frac{Y}{X}\rb)^5 -8\lb\{\lb(\frac{X}{Y}\rb)^3+\lb(\frac{Y}{X}\rb)^3\rb\}\\&+4\lb(\frac{X}{Y}+\frac{Y}{X}\rb) +4\big/\lb(\frac{X}{Y}+\frac{Y}{X}\rb).
\end{split}\end{equation}
\end{lemma}
\begin{lemma}\cite{MSM5}
 If $P:=\dfrac{\varphi(-q)}{\varphi(-q^5)}$ and $Q:=\dfrac{\varphi(-q^4)}{\varphi(-q^{20})},$
then
\begin{equation}\begin{split}\label{b12}
&\frac{P^4}{Q^4}+\frac{Q^4}{P^4}+24\lb[\frac{P^2}{Q^2}+\frac{Q^2}{P^2}\rb] +8\lb[P^2Q^2+\frac{5^2}{P^2Q^2}\rb]+3\lb[Q^4+\frac{5^2}{Q^4}\rb]+120\\&= 20\lb[P^2+\frac{5}{P^2}\rb]+32\lb[Q^2+\frac{5}{Q^2}\rb] +\lb[P^2Q^4+\frac{5^4}{P^2Q^4}\rb]+3\lb[\frac{5P^2}{Q^4}+\frac{Q^4}{P^2}\rb].
\end{split}\end{equation}
\end{lemma}
\begin{lemma}\cite[Theorem 5.3]{CA3}
 If $U:=\dfrac{\varphi^2(q)}{\varphi^2(q^5)}$ and $V:=\dfrac{\psi^2(-q)}{q\psi^2(-q^5)}$, then
\begin{equation}\label{b3}
U+UV=5+V.
\end{equation}
\end{lemma}
\begin{lemma}\cite[Theorem 2.17]{NDB}
 If $U:=\dfrac{\varphi(-q)}{\varphi(-q^5)}$ and $V:=\dfrac{\varphi(-q^2)}{\varphi(-q^{10})}$, then
\begin{equation}\label{b9}
\frac{U^2}{V^2}+\frac{V^2}{U^2}+4=V^2+\frac{5}{V^2}.
\end{equation}
\end{lemma}
\newpage
\section{$P$-$Q$ modular equations of degree 5}\label{sec3}
In this section, we establish some new modular equations of degree 5.
\begin{theorem}
If $M:=\dfrac{\varphi(-q)}{\varphi(-q^5)}$ and $N:=\dfrac{\varphi(-q^6)}{\varphi(-q^{30})}$, then
\begin{equation}\begin{split}\label{b22}
&\lb[\frac{M^8}{N^8}+\frac{N^8}{M^8}\rb]+96\lb[\frac{M^6}{N^6}+\frac{N^6}{M^6}\rb] +1146\lb[\frac{M^4}{N^4}+\frac{N^4}{M^4}\rb]+2868\lb[\frac{M^2}{N^2}+\frac{N^2}{M^2}\rb] \\&+\lb[N^8+\frac{5^4}{N^8}\rb]-16\lb[N^6+\frac{5^3}{N^6}\rb]+188\lb[N^4+\frac{5^2}{N^4}\rb] -1696\lb[N^2+\frac{5}{N^2}\rb]\\&-54\lb[M^6+\frac{5^3}{M^6}\rb]+498\lb[M^4+\frac{5^2}{M^4}\rb] -2106\lb[M^2+\frac{5}{M^2}\rb]-4\lb[\frac{N^8}{M^6}+\frac{5M^6}{N^8}\rb]\\& +6\lb[\frac{N^8}{M^4}+\frac{5^2M^4}{N^8}\rb]-4\lb[\frac{N^8}{M^2}+\frac{5^3M^2}{N^8}\rb] -144\lb[\frac{N^6}{M^4}+\frac{5M^4}{N^6}\rb]+64\lb[\frac{N^6}{M^2}+\frac{5^2M^2}{N^6}\rb] \\&-479\lb[\frac{N^4}{M^2}+\frac{5M^2}{N^4}\rb]-165\lb[\frac{M^6}{N^4}+\frac{5N^4}{M^6}\rb] +124\lb[\frac{M^6}{N^2}+\frac{5^2N^2}{M^6}\rb]-936\lb[\frac{M^4}{N^2}+\frac{5N^2}{M^4}\rb] \\&+10\lb[M^4N^4+\frac{5^4}{M^4N^4}\rb]+516\lb[M^2N^2+\frac{5^2}{M^2N^2}\rb] -39\lb[M^2N^4+\frac{5^3}{M^2N^4}\rb]\\&-120\lb[M^4N^2+\frac{5^3}{M^4N^2}\rb] +12\lb[M^6N^2+\frac{5^4}{M^6N^2}\rb]-\lb[M^6N^4+\frac{5^5}{M^6N^4}\rb]+6748=0.
\end{split}\end{equation}
\end{theorem}
\begin{proof}
Using the equation \eqref{b7} after changing $q$ to $q^2$, we get
\begin{equation}\label{b19}
X^9Y^9+125X^3Y^3+X^{12}-Y^{12}+9X^9Y^3+9Y^9X^3=0.
\end{equation}
where \begin{equation*}
X:=\dfrac{f(-q^2)}{q^{1/3}f(-q^{10})} \,\,\,\textrm{ and }\,\,\, Y:=\dfrac{f(-q^6)}{qf(-q^{30})}.
\end{equation*}
Cubing the equation \eqref{b19} and using the equations \eqref{b10} and \eqref{b11}, we deduce
\begin{equation}\label{b21}
TM^3R^3N^6T_2+125MRN^2T+M^4R^4-N^8T_2^2+9M^3R^3N^2T+9TN^6T_2MR=0.
\end{equation}
where \begin{equation*}
R:=\dfrac{\psi^2(q)}{q\psi^2(q^5)},\ T:=\dfrac{\psi(q^6)}{q^{3}\psi(q^{30})} \,\,\,\textrm{ and }\,\,\, T_2:=\dfrac{\psi^2(q^6)}{q^{6}\psi^2(q^{30})}.
\end{equation*}
Using the equation \eqref{b3} after changing $q$ to $-q$, we have
\begin{equation}\label{b20}
R:=\dfrac{M^2-5}{M^2-1}\,\,\, \textrm{and}\,\,\, T_2:=\dfrac{N^2-5}{N^2-1}.
\end{equation}
Collecting the terms containing $T$ on one side of the equation \eqref{b21} and using the equation \eqref{b20}, we get
\begin{equation}\label{b24}
A(M,N)B(M.N)=0,
\end{equation}
where
\begin{equation*}\begin{split}
&A(M,N):=\left(625-500M^2-20N^6-500N^2+150M^4+N^8+150N^4\rb.\\&\lb.+M^8-20M^6 +16M^6N^6+400M^2N^2+120M^4N^2+24M^6N^4+24M^4N^6\rb.\\&\lb.-300M^4N^4-16M^2N^6- 16M^6N^2-4M^8N^2+N^8M^8-4N^8M^6\rb.\\&\lb.+6N^8M^4-4N^8M^2-4M^8N^6+120M^2N^4+6M^8N^4\right)
\end{split}\end{equation*} and
\begin{equation*}\begin{split}
&B(M,N):=\left(625M^8+N^{16}-825N^{12}M^2+150M^{12}-500M^{10}+M^{16}\rb.\\&\lb.+12900M^6N^6 -4875M^6N^4-15000M^4N^6+6250M^4N^4+7500M^2N^6\rb.\\&\lb.-2000M^8N^2+1146M^{12}N^4-720M^{12}N^2 -2395M^{10}N^4+1600M^{10}N^2\rb.\\&\lb.+188N^{12}M^8-479N^{12}M^6+1146N^{12}M^4-1696N^{10}M^8 +2868N^{10}M^6\rb.\\&\lb.-4680N^{10}M^4+3100N^{10}M^2+6748N^8M^8-10530N^8M^6+12450N^8M^4 \rb.\\&\lb.-6750N^8M^2+10M^{12}N^{12}-120M^{12}N^{10}-39M^{10}N^{12}+516M^{10}N^{10}\rb.\\&\lb.-2106M^{10}N^8 +2868M^{10}N^6-8480M^8N^6+498M^{12}N^8-936M^{12}N^6\rb.\\&\lb.-165M^{14}N^4+96M^{14}N^2 -M^{14}N^{12}+12M^{14}N^{10}-54M^{14}N^8-20M^{14}\rb.\\&\lb.+124M^{14}N^6+96N^{14}M^2-16N^{14}M^8+64N^{14}M^6 -144N^{14}M^4\rb.\\&\lb.-4N^{16}M^2+N^{16}M^8-4N^{16}M^6+6N^{16}M^4-3125M^2N^4+4700M^8N^4\right).
\end{split}\end{equation*}
Expanding in powers of $q$, the first and second factor of the equation \eqref{b24}, one gets respectively,
$$A(M,N)=\left(256-1536q^8-512q^9+1152q^{10}+3840q^{11}+4736q^{12}+\dotsb\right)$$ and
$$B(M,N)=q^8\left(8448+33792q+33792q^2-54528q^3-194208q^4-268608q^5+\dotsb\right).$$
As $q\rightarrow0$, the factor $B(M,N)$ of the equation \eqref{b24} vanishes whereas the other factor $A(M,N)$ do not vanish. Hence, we arrive at the equation \eqref{b22} for $q\in(0,1)$. By analytic continuation the equation \eqref{b22} is true for $|q|<1$.
\end{proof}
\begin{remark}
The modular relation connecting
\begin{equation*}
\dfrac{\psi(q)}{q^{1/2}\psi(q^5)}\,\,\,\textrm{ and }\,\,\,\dfrac{\psi(q^6)}{q^{3}\psi(q^{30})},
\end{equation*}
can be obtained by eliminating $M$ and $N$ from the equation \eqref{b21}.
\end{remark}
\section{$P$--$Q$ ``mixed'' modular equations}\label{sec4}
In this section, we establish several new $P$--$Q$ ``mixed'' modular equations with four moduli. Throughout this section, we set
\begin{equation}\label{c1}\begin{split}
&A:=\frac{f(-q)f(-q^2)}{q^{1/2}f(-q^5)f(-q^{10})},\,\,\, B_{n}:=\frac{f(-q^n)f(-q^{2n})}{q^{n/2}f(-q^{5n})f(-q^{10n})}\\&
\,\,\,\,\,\,\,\,\,\,\,\,\textrm {and} \,\,\,\,\,\, C_{n}:=\frac{q^{{n}/{6}}f(-q^n)f(-q^{10n})}{f(-q^{2n})f(-q^{5n})}.
\end{split}\end{equation}
\begin{theorem}For $|q|<1,$
\begin{eqnarray}
&&\frac{f^2(-q)f^2(-q^2)}{qf^2(-q^5)f^2(-q^{10})} =\frac{U(U-5)}{U-1},\,\,\,U>1,\label{c8}\\
&&\frac{f^2(-q)f^2(-q^2)}{qf^2(-q^5)f^2(-q^{10})}=\frac{V(V-5)}{V-1},\,\,\,V>1,\label{c9}
\end{eqnarray}
where $\displaystyle U:=\frac{\varphi^2(-q)}{\varphi^2(-q^5)}$  $\,\,\textrm{and}\,\,\, \displaystyle V:=\frac{\psi^2(q)}{q\psi^2(q^5)}.$
\end{theorem}
\begin{proof}[Proof of \eqref{c8}]
The equations \eqref{b4} and \eqref{b5} can be rewritten as
\begin{equation}\label{b6}
m\left\{\frac{\alpha(1-\alpha)}{\beta(1-\beta)}\right\}^{1/4}+1=\frac{5}{m} +\left\{\frac{\alpha(1-\alpha)}{\beta(1-\beta)}\right\}^{1/4}.
\end{equation}
Employing the equation \eqref{b1} after changing $q$ to $-q$  and equation \eqref{b2} in the equation \eqref{b6}, we arrive at the equation \eqref{c8}.
\end{proof}
\begin{proof}[Proof of \eqref{c9}]
Using the equation \eqref{b3} in the equation \eqref{c8}, we arrive at the equation \eqref{c9}.
\end{proof}
\begin{theorem}For $|q|<1,$
\begin{eqnarray}
&&\frac{qf^6(-q)f^6(-q^{10})}{f^6(-q^2)f^6(-q^{5})} =\frac{U(U-1)}{U-5},\,\,\,\label{c10}\\
&&\frac{qf^6(-q)f^6(-q^{10})}{f^6(-q^2)f^6(-q^{5})} =\frac{(V-5)}{V(V-1)},\,\,\,\label{c11}
\end{eqnarray}
where $\displaystyle U:=\frac{\varphi^2(-q)}{\varphi^2(-q^5)}$  $\,\,\textrm{and}\,\,\, \displaystyle V:=\frac{\psi^2(q)}{q\psi^2(q^5)}.$
\end{theorem}
\begin{proof}
The proof of equations \eqref{c10} and \eqref{c11} are similar to the proof of \eqref{c8} and \eqref{c9}. Hence, we omit the details.
\end{proof}
\begin{theorem}
If $P:=AB_{2}$ and $Q:=\dfrac{A}{B_{2}}$, then
\begin{equation}\label{c2}
\lb(Q^4+\frac{1}{Q^4}\rb)-3\lb(Q^2+\frac{1}{Q^2}\rb)-\lb(P+\frac{5^2}{P}\rb)\lb(Q+\frac{1}{Q}\rb)-12=0.
\end{equation}
\end{theorem}
\begin{proof}
Taking the cube of both sides of the equation \eqref{b16}, we deduce
 \begin{equation}\label{c29}\begin{split}
 X^3Y^3+\frac{125}{X^3Y^3}+15\left(\frac{X^3}{Y^3}+\frac{Y^3}{X^3}\right) =\frac{X^9}{Y^9}+\frac{Y^9}{X^9},
 \end{split}\end{equation}
 where $X:=\dfrac{f(-q)}{q^{1/6}f(-q^5)}$ and $Y:=\dfrac{f(-q^2)}{q^{1/3}f(-q^{10})}$.
 Using \eqref{b10}, \eqref{b11} and \eqref{c1}, we deduce
 \begin{equation}\label{b30}\begin{split}
 &A^4V_1^4B_2^4V_2^4+125A^2V_1^2B_2^2V_2^2+12A^4V_1^4B_2^2V_2^2 +12B_2^4V_2^4A^2V_1^2\\&=A^6V_1^6+B_2^6V_2^6,
 \end{split}\end{equation}
 where $V_1:=\dfrac{\psi(q)}{q^{1/2}\psi(q^5)}$\,\,and\,\, $V_2:=\dfrac{\psi(q^2)}{q\psi(q^{10})}$.\\
 Using the equation \eqref{c9} in the equation \eqref{b30}, we deduce
 \begin{equation}\begin{split}\label{b31}
 &12B_{2}^6A^2vu+A^6B_{2}^6uv+5A^6B_{2}^4uv+5A^4B_{2}^6uv+12A^6B_{2}^2uv+625A^2B_{2}^2v\\&+625A^2B_{2}^2u-18A^8u -18B_{2}^8v+185A^4B_{2}^4v+185A^4B_{2}^4u+300A^4B_{2}^2u\\&+60A^4B_{2}^2uv+300B_{2}^4A^2v+60B_{2}^4A^2vu+A^8B_{2}^6v -250A^6-250B_{2}^6\\&+1350A^4B_{2}^4+2125A^4B_{2}^2+A^6B_{2}^8u+37A^4B_{2}^6v+40A^6B_{2}^4v+8A^6B_{2}^6v\\&+40A^4B_{2}^6u +37A^6B_{2}^4u+5A^4B_{2}^8u+8A^6B_{2}^6u+60A^6B_{2}^2u+60B_{2}^6A^2v\\&+425A^4B_{2}^2v-50A^6u-50B_{2}^6v+12A^2B_{2}^8u +5A^8B_{2}^4v+425A^2B_{2}^4u\\&+12A^8B_{2}^2v+125A^2B_{2}^2uv+64A^6B_{2}^6+96A^6B_{2}^2v+96B_{2}^6A^2u+25A^4B_{2}^4uv \\&-2B_{2}^{14}v-2A^{14}u+2125A^2B_{2}^4+3125A^2B_{2}^2+A^8B_{2}^8+8A^8B_{2}^6+37A^8B_{2}^4\\&+8A^6B_{2}^8+296A^6B_{2}^4 +37A^4B_{2}^8+296A^4B_{2}^6+60A^8B_{2}^2+480A^6B_{2}^2\\&+60A^2B_{2}^8+480A^2B_{2}^6-120B_{2}^8-120A^8-24B_{2}^{14} -24A^{14}-2A^{12}\\&-2B_{2}^{12}=0.
 \end{split}\end{equation}
 where $u:=\pm\sqrt{A^4+6A^2+25}$ \,\, and \,\,  $v:=\pm\sqrt{B_{2}^4+6B_{2}^2+25}$.\\
 Eliminating $u$ and $v$ from  the equation \eqref{b31}, we find  and squaring both sides, we deduce
 \begin{equation*}\begin{split}
&\lb(A^8-A^6B_{2}^4-3A^6B_{2}^2-12A^4B_{2}^4-25A^4B_{2}^2-A^4B_{2}^6-3A^2B_{2}^6-25A^2B_{2}^4\rb.\\&\lb.+B_{2}^8\rb) \lb(A^{16}+B_{2}^{16}-187500A^6B_{2}^4-187500A^4B_{2}^6-26900A^6B_{2}^8\rb.\\&\lb.-390625A^4B_{2}^4 -93125A^6B_{2}^6-57500A^4B_{2}^8-15625A^2B_{2}^8-9600B_{2}^{10}A^4\rb.\\&\lb.-5000B_{2}^{10}A^2-875B_{2}^{12}A^2 -7676A^8B_{2}^8-57500A^8B_{2}^4-15625A^8B_{2}^2\rb.\\&\lb.-26900A^8B_{2}^6-1076A^8B_{2}^{10}-92A^8B_{2}^{12} -A^8B_{2}^{14}-35A^{14}B_{2}^4-384A^{12}B_{2}^6\rb.\end{split}\end{equation*}\begin{equation}\begin{split}\label{b33}&\lb.-A^{14}B_{2}^8-A^{12}B_{2}^{12}-52A^{14}B_{2}^2-1024A^{12}B_{2}^4 -875A^{12}B_{2}^2-3987A^{10}B_{2}^6\rb.\\&\lb.-9600A^{10}B_{2}^4-8A^{14}B_{2}^6-5000A^{10}B_{2}^2-92A^{12}B_{2}^8 -1076A^{10}B_{2}^8\rb.\\&\lb.-149A^{10}B_{2}^{10}-12A^{10}B_{2}^{12}-384A^6B_{2}^{12} -3987A^6B_{2}^{10}-8A^6B_{2}^{14}\rb.\\&\lb.-35A^4B_{2}^{14}-12A^{12}B_{2}^{10} -1024B_{2}^{12}A^4-52A^2B_{2}^{14}\rb)=0.
\end{split}\end{equation}
Expanding in powers of $q$, the first and second factors of \eqref{b33}, one gets respectively
$$-4q^{11}\lb(4+8q-27q^2-68q^3+40q^4+278q^5+62q^6-723q^7+\cdots\rb)$$ and
$$\lb(-2+2q^2-10q^3+30q^4+552q^5-2016q^6+1038q^7+15620q^8+\cdots\rb).$$
As $q$ tends to 0 the first factor of \eqref{b33} vanishes whereas the second factor does not vanish. Hence we arrive at \eqref{c2} for $q\in(0,1)$. By analytic continuation \eqref{c2} is true for $|q|<1$.
\end{proof}
\begin{theorem}
If $P=AB_{4}$ and $Q=\dfrac{A}{B_{4}}$, then
\begin{equation}\begin{split}\label{b34}
&\lb(Q^8+\frac{1}{Q^8}\rb)-52\lb(Q^6+\frac{1}{Q^6}\rb)-1024\lb(Q^4+\frac{1}{Q^4}\rb) -3987\lb(Q^2+\frac{1}{Q^2}\rb)\\&-\lb(P+\frac{5^2}{P}\rb)\lb[1076\lb(Q+\frac{1}{Q}\rb) +384\lb(Q^3+\frac{1}{Q^3}\rb)+35\lb(Q^5+\frac{1}{Q^5}\rb)\rb]\\&-\lb(P^2+\frac{5^4}{P^2}\rb) \lb[92\lb(Q^2+\frac{1}{Q^2}\rb)+8\lb(Q^4+\frac{1}{Q^4}\rb)+149\rb]-\lb(P^3+\frac{5^6}{P^3}\rb) \\&\times\lb[12\lb(Q+\frac{1}{Q}\rb)+\lb(Q^3+\frac{1}{Q^3}\rb)\rb]-\lb(P^4+\frac{5^8}{P^4}\rb)-7676=0.
\end{split}\end{equation}
\end{theorem}
\begin{proof}
Using the equation \eqref{b12} in the equation \eqref{c8}, we deduce
\begin{equation*}\begin{split}
&625-125u+125v-A^4uB_{4}^6v-5A^4uB_{4}^4v-90A^2uB_{4}^2v-6A^2uB_{4}^6v\\&-38A^2uB_{4}^4v +40A^2u+2A^6u-1215B_{4}^2v-63B_{4}^6v-A^6v-56A^6B_{4}^2\\&-36A^6B_{4}^4-8A^6B_{4}^6 -A^6B_{4}^8+3A^4v-592A^4B_{4}^2-360A^4B_{4}^4-80A^4B_{4}^6\\&-9A^4B_{4}^8-15A^2v -3080A^2B_{4}^2-1732A^2B_{4}^4-376A^2B_{4}^6-39A^2B_{4}^8\\&-25uv-1360uB_{4}^2 -688uB_{4}^4-144uB_{4}^6-13uB_{4}^8-125A^2-6000B_{4}^2\\&+29A^4-3A^6+2A^8-3216B_{4}^4 -688B_{4}^6-63B_{4}^8-9A^4u-13A^6B_{4}^2v\\&-A^6B_{4}^6v-5A^6B_{4}^4v-129A^4B_{4}^2v -9A^4B_{4}^6v-53A^4B_{4}^4v-A^4uv\end{split}\end{equation*}\begin{equation}\begin{split}\label{b36}&-56A^4uB_{4}^2-36A^4uB_{4}^4-8A^4uB_{4}^6-A^4uB_{4}^8 -643A^2B_{4}^2v-39A^2B_{4}^6v\\&-259A^2B_{4}^4v+6A^2uv-424A^2uB_{4}^2-252A^2uB_{4}^4 -56A^2uB_{4}^6-6A^2uB_{4}^8\\&-269uB_{4}^2v-13uB_{4}^6v-105uB_{4}^4v-13A^4uB_{4}^2v -499B_{4}^4v=0.
\end{split}\end{equation}
where $u:=\pm\sqrt{A^4+6A^2+25}$ \,\, and \,\,  $v:=\pm\sqrt{B_{4}^4+6B_{4}^2+25}$.\\
 Eliminating $u$ and $v$ from the equation \eqref{b36}, we arrive at \eqref{b34}.
\end{proof}
\begin{theorem}
If $P=AB_{6}$ and $Q=\dfrac{A}{B_{6}}$, then
\begin{equation*}\begin{split}
&\mathbb{Q}^{16}-363\mathbb{Q}^{14}-30882\mathbb{Q}^{12}-698682\mathbb{Q}^{10}-6183702\mathbb{Q}^{8}-16140317 \mathbb{Q}^{6}\\&+37225608\mathbb{Q}^{4}+231497788 \mathbb{Q}^{2} +\mathbb{P}\lb\{60133800\mathbb{Q}+21753498\mathbb{Q}^{3}\rb.\\&\lb.-1148442\mathbb{Q}^{5}-2210604\mathbb{Q}^{7}-406488\mathbb{Q}^{9} -26740\mathbb{Q}^{11}-519\mathbb{Q}^{13}\rb\}\\&+\mathbb{P}^{2}\lb\{6287236\mathbb{Q}^{2}+858465\mathbb{Q}^{4} -462222\mathbb{Q}^{6}-150099\mathbb{Q}^{8}-12840\mathbb{Q}^{10} \rb.\\&\lb. -267\mathbb{Q}^{12}+10229305\rb\}+\mathbb{P}^{3} \lb\{1132002\mathbb{Q}+362832\mathbb{Q}^{3}-42462\mathbb{Q}^{5}-37066\mathbb{Q}^{7}\rb.\\&\lb.-4323\mathbb{Q}^{9}-78\mathbb{Q}^{11}\rb\} +\mathbb{P}^{4}\lb\{74418\mathbb{Q}^{2}+4471\mathbb{Q}^{4}-5955\mathbb{Q}^{6}-1026\mathbb{Q}^{8}-12\mathbb{Q}^{10} \rb.\end{split}\end{equation*}\begin{equation}\begin{split}\label{b35}&\lb. +130902\rb\}+\mathbb{P}^{5}\lb\{9171\mathbb{Q}+2028\mathbb{Q}^{3}-588\mathbb{Q}^{5}-171\mathbb{Q}^{7}-\mathbb{Q}^{9}\rb\} +\mathbb{P}^{6}\lb\{300\mathbb{Q}^{2}\rb.\\&\lb.-27\mathbb{Q}^{4}-18\mathbb{Q}^{6}+679\rb\}+\mathbb{P}^{7}\lb\{24\mathbb{Q} -\mathbb{Q}^{5}\rb\}+\mathbb{P}^{8}+36965548 =0,
\end{split}\end{equation}
\end{theorem}
where $\mathbb{P}^{n}=\lb(P^{n}+\dfrac{5^{2n}}{P^{n}}\rb)$ and $\mathbb{Q}^{n}=\lb(Q^{n}+\dfrac{1}{Q^{n}}\rb)$.
\begin{proof}
The proof of the equation \eqref{b35} is similar  to the proof of the equation \eqref{b34}; Notice that now \eqref{b22} is used in place of \eqref{b12}.
\end{proof}
\begin{theorem}
If $P=C_{1}C_{2}$ and $Q=\dfrac{C_{1}}{C_{2}}$, then
\begin{equation}\label{b28}
\lb(P+\frac{1}{P}\rb)\lb(Q^3+\frac{1}{Q^3}\rb)+2=\lb(P^2+\frac{1}{P^2}\rb).
\end{equation}
\end{theorem}
\begin{proof}
Using the equation \eqref{b9} in the equation \eqref{c10}, we deduce
\begin{equation}\begin{split}\label{b25}
&10v-10u-2vC_{2}^6C_{1}^6-2vuC_{2}^6+6vu+4uC_{1}^6+2vC_{2}^6-2C_{2}^{12}C_{1}^6 -6 \\&+24C_{2}^6C_{1}^6-2uC_{2}^{12}+24uC_{2}^6+6vC_{1}^6-46C_{1}^6-8C_{2}^6+4C_{1}^{12} +2C_{2}^{12}=0.
\end{split}\end{equation}
where $u:=\pm\sqrt{C_{1}^{12}-18C_{1}^6+1}$ \,\, and \,\, $v:=\pm\sqrt{C_{2}^{12}-18C_{2}^6+1}.$\\

Eliminating $u$ and $v$ from the equation \eqref{b25} leads to
\begin{equation}\begin{split}\label{b27}
&\lb(C_{1}^6-C_{2}^6C_{1}^6+C_{2}^6-C_{2}^2C_{1}^2+C_{2}^2C_{1}^8+2C_{2}^4C_{1}^4 +C_{2}^8C_{1}^2\rb)\lb(C_{2}^4C_{1}^{16}\rb.\\&\lb.+C_{2}^8C_{1}^{14}-C_{2}^2C_{1}^{14} +C_{2}^{12}C_{1}^{12}-4C_{2}^6C_{1}^{12}+C_{1}^{12}+4C_{2}^{10}C_{1}^{10} -4C_{2}^4C_{1}^{10}\rb.\\&\lb.+C_{2}^{14}C_{1}^8+C_{2}^8C_{1}^8+C_{2}^2C_{1}^8 -4C_{2}^{12}C_{1}^6+4C_{2}^6C_{1}^6+C_{2}^{16}C_{1}^4-4C_{2}^{10}C_{1}^4 \rb.\\&\lb.+C_{2}^4C_{1}^4-C_{2}^{14}C_{1}^2+C_{2}^8C_{1}^2+C_{2}^{12}\rb)=0.
\end{split}\end{equation}
Expanding in powers of $q$, the first and second factors of \eqref{b27}, one gets respectively
$$q^{11}\lb(8-32q-8q^2+168q^3-220q^4+196q^5-760q^6+1748q^7+\cdots\rb)$$ and
$$\lb(3-24q+117q^2-456q^3+1356q^4-3192q^5+7242q^6-17304q^7+\cdots\rb).$$
As $q$ tends to 0 the first factor of \eqref{b27} vanishes whereas the second factor does not vanish. Hence we arrive at \eqref{b28} for $q\in(0,1)$. By analytic continuation \eqref{b28} is true for $|q|<1$.
\end{proof}
\begin{theorem}
If $P=C_{1}C_{4}$ and $Q=\dfrac{C_{1}}{C_{4}}$, then
\begin{equation}\begin{split}\label{b29}
&\lb(P^3+\frac{1}{P^3}\rb)\lb[19\lb(Q+\frac{1}{Q}\rb)+8\lb(Q^3+\frac{1}{Q^3}\rb) +\lb(Q^5+\frac{1}{Q^5}\rb)\rb]\\&+\lb(Q^6+\frac{1}{Q^6}\rb)+13\lb(Q^4+\frac{1}{Q^4}\rb) +52\lb(Q^2+\frac{1}{Q^2}\rb)+82=\lb(P^6+\frac{1}{P^6}\rb).
\end{split}\end{equation}
\end{theorem}
\begin{proof}
The proof of the equation \eqref{b29} is similar to the proof of \eqref{b28};Notice that now \eqref{b12} is used in place of \eqref{b9}.
\end{proof}
 \section{Remarkable product of theta-functions}\label{sec5}
In this section, we establish several new modular identities connecting the remarkable product of theta-functions $b_{s,5}$ with $b_{r^2s,5}$ for $r=$ 2, 4, and 6.
\begin{lemma}\cite{MSMMCMKSB}
If $s$ and $t$ are any positive rational, then
\begin{equation}\label{d1}
 b_{2s,t}b_{\frac{2}{s},t}=1.
\end{equation}
\end{lemma}
\begin{lemma}\cite{MSMCKSHM}
$0<b_{s,t}\leq1$ for all $s\geq2$ and $t$ positive integer greater than 1.
\end{lemma}
\begin{theorem}
If $X=\sqrt{b_{s,5}b_{4s,5}}$ and $Y=\displaystyle\sqrt{\frac{b_{s,5}}{b_{4s,5}}}$, then
\begin{equation}\label{d2}
\lb(Y^4+\frac{1}{Y^4}\rb)-3\lb(Y^2+\frac{1}{Y^2}\rb)-5\lb(X+\frac{1}{X}\rb)\lb(Y+\frac{1}{Y}\rb)-12=0.
\end{equation}
\end{theorem}
\begin{proof}
Using the equation \eqref{bmn} in the equation \eqref{c2} we arrive at the equation \eqref{d2}.
\end{proof}
\begin{corollary}
\begin{equation}\label{d6}
b_{4,5}=\frac{\sqrt{2+2\sqrt{5}-2\sqrt{2+2\sqrt{5}}}}{2},
\end{equation}
\begin{equation}\label{d6a}
b_{1,5}=\frac{\sqrt{2+2\sqrt{5}+2\sqrt{2+2\sqrt{5}}}}{2}.
\end{equation}
\end{corollary}
\begin{proof}
Putting $s=1/2,$ in \eqref{d2} and using the fact that $b_{1,5}b_{4,5}=1$, we deduce
\begin{equation}\label{d3}
(h^8-2h^6-2h^4-2h^2+1)(h^2+h+1)(h^2-h+1)=0,
\end{equation}
where $h:=b_{4,5}.$\\
We observe that the first factor of  \eqref{d3} vanishes for specific value of $q:=e^{-\pi\sqrt{4/5}}$, whereas the other factors does not vanish. Hence, we have
\begin{equation}\label{d4}
t^2-2t-4=0,
\end{equation}
where $t:=h^2+\displaystyle\frac{1}{h^2}.$\\
On solving the equation \eqref{d4} for $h$ and $t>0$, we deduce
\begin{equation}\label{d5}
h^2+\displaystyle\frac{1}{h^2}=1+\sqrt{5}.
\end{equation}
On solving the equation \eqref{d5} for $h$ and $0<h<1$, we arrive at \eqref{d6} and \eqref{d6a}.
\end{proof}
\begin{theorem}
If $X=\sqrt{b_{s,5}b_{16s,5}}$ and $Y=\displaystyle\sqrt{\frac{b_{s,5}}{b_{16s,5}}}$, then
\begin{equation}\begin{split}\label{d7}
&\lb(Y^8+\frac{1}{Y^8}\rb)-52\lb(Y^6+\frac{1}{Y^6}\rb)-1024\lb(Y^4+\frac{1}{Y^4}\rb) -3987\lb(Y^2+\frac{1}{Y^2}\rb)\\&-5\lb(X+\frac{1}{X}\rb)\lb[1076\lb(Y+\frac{1}{Y}\rb) +384\lb(Y^3+\frac{1}{Y^3}\rb)+35\lb(Y^5+\frac{1}{Y^5}\rb)\rb]\\&-5^2\lb(X^2+\frac{1}{X^2}\rb) \lb[92\lb(Y^2+\frac{1}{Y^2}\rb)+8\lb(Y^4+\frac{1}{Y^4}\rb)+149\rb]-5^3\lb(X^3+\frac{1}{X^3}\rb) \\&\times\lb[12\lb(Y+\frac{1}{Y}\rb)+\lb(Y^3+\frac{1}{Y^3}\rb)\rb]-5^4\lb(X^4+\frac{1}{X^4}\rb)-7676=0.
\end{split}\end{equation}
\end{theorem}
\begin{proof}
Using the equation \eqref{bmn} in the equation \eqref{b34} we arrive at the equation \eqref{d7}.
\end{proof}
\begin{corollary}
\begin{equation}\label{d11}
b_{8,5}=\sqrt{(\sqrt{2}-1)(\sqrt{5}-2)},
\end{equation}
\begin{equation}\label{d12}
b_{1/2,5}=\sqrt{(\sqrt{2}+1)(\sqrt{5}+2)}.
\end{equation}
\end{corollary}
\begin{proof}
Putting $s=1/4,$ in \eqref{d7} and using the fact that $b_{1/2,5}b_{8,5}=1$, we deduce
\begin{equation}\label{d8}\begin{split}
&\lb(h^8-8h^6-22h^4-8h^2+1\rb)\lb(h^4+3h^2+1\rb)\lb(h^4-h^3+h^2+h+1\rb) \\&\lb(h^4+h^3+h^2-h+1\rb)=0,
\end{split}\end{equation}
where $h:=b_{8,5}.$\\
We observe that the first factor of  \eqref{d8} vanishes for specific value of $q:=e^{-\pi\sqrt{8/5}}$, whereas the other factors does not vanish. Hence, we have
\begin{equation}\label{d9}
t^2-8t-24=0,
\end{equation}
where $t:=h^2+\displaystyle\frac{1}{h^2}.$\\
On solving the equation \eqref{d9} for $h$ and $t>0$, we deduce
\begin{equation}\label{d10}
h^2+\displaystyle\frac{1}{h^2}=4+2\sqrt{10}.
\end{equation}
On solving the equation \eqref{d10} for $h$ and $0<h<1$, we arrive at \eqref{d11} and \eqref{d12}.
\end{proof}
\begin{theorem}
If $X=\sqrt{b_{s,5}b_{36s,5}}$ and $Y=\displaystyle\sqrt{\frac{b_{s,5}}{b_{36s,5}}}$, then
\begin{equation}\begin{split}\label{d15}
&\mathbb{Y}^{16}-363\mathbb{Y}^{14}-30882\mathbb{Y}^{12}-698682\mathbb{Y}^{10}-6183702\mathbb{Y}^{8}-16140317 \mathbb{Y}^{6}\\&+37225608\mathbb{Y}^{4}+231497788 \mathbb{Y}^{2} +5\mathbb{X}\lb\{60133800\mathbb{Y}+21753498\mathbb{Y}^{3}\rb.\\&\lb.-1148442\mathbb{Y}^{5}-2210604\mathbb{Y}^{7}-406488\mathbb{Y}^{9} -26740\mathbb{Y}^{11}-519\mathbb{Y}^{13}\rb\}\\&+5^2\mathbb{X}^{2}\lb\{6287236\mathbb{Y}^{2}+858465\mathbb{Y}^{4} -462222\mathbb{Y}^{6}-150099\mathbb{Y}^{8}-12840\mathbb{Y}^{10} \rb.\\&\lb. -267\mathbb{Y}^{12}+10229305\rb\}+5^3\mathbb{X}^{3} \lb\{1132002\mathbb{Y}+362832\mathbb{Y}^{3}-42462\mathbb{Y}^{5}\rb.\\&\lb.-37066\mathbb{Y}^{7}-4323\mathbb{Y}^{9}-78\mathbb{Y}^{11}\rb\} +5^4\mathbb{X}^{4}\lb\{74418\mathbb{Y}^{2}+4471\mathbb{Y}^{4}-5955\mathbb{Y}^{6}\rb.\\&\lb.-1026\mathbb{Y}^{8}-12\mathbb{Y}^{10} +130902\rb\}+5^5\mathbb{X}^{5}\lb\{9171\mathbb{Y}+2028\mathbb{Y}^{3}-588\mathbb{Y}^{5}-171\mathbb{Y}^{7}\rb.\\&\lb.-\mathbb{Y}^{9}\rb\} +5^6\mathbb{X}^{6}\lb\{300\mathbb{Y}^{2}-27\mathbb{Y}^{4}-18\mathbb{Y}^{6}+679\rb\}+5^7\mathbb{X}^{7}\lb\{24\mathbb{Y} -\mathbb{Y}^{5}\rb\}+5^8\mathbb{X}^{8}\\&+36965548 =0,
\end{split}\end{equation}
\end{theorem}
where $\mathbb{X}^{n}=\lb(X^{n}+\dfrac{1}{X^{n}}\rb)$ and $\mathbb{Y}^{n}=\lb(Y^{n}+\dfrac{1}{Y^{n}}\rb)$.
\begin{proof}
Using the equation \eqref{bmn} in the equation \eqref{b35} we arrive at the equation \eqref{d15}.
\end{proof}
\begin{corollary}
\begin{equation}\label{d13}
b_{12,5}=\sqrt{\frac{(2-\sqrt{3})(7-3\sqrt{5})}{2}},
\end{equation}
\begin{equation}\label{d14}
b_{1/3,5}=\sqrt{\frac{(2+\sqrt{3})(7+3\sqrt{5})}{2}}.
\end{equation}
\end{corollary}
\begin{proof}
Putting $s=1/6,$ in \eqref{d15} and using the fact that $b_{1/3,5}b_{12,5}=1$, we deduce,
\begin{equation}\label{d16}\begin{split}
&\lb(h^8-28 h^6+63 h^4-28 h^2+1\rb)\lb(h^{12}+10 h^{10}+15 h^8+28 h^6+15 h^4+10 h^2+1\rb) \\&\lb(h^8-2 h^7+4 h^6-h^5+7 h^4+h^3+4 h^2+2 h+1\rb)\lb(h^8+2 h^7+4 h^6+h^5+7 h^4\rb.\\&\lb.-h^3+4 h^2-2 h+1\rb)=0,
\end{split}\end{equation}
where $h:=b_{12,5}.$\\
We observe that the first factor of  \eqref{d16} vanishes for specific value of $q:=e^{-\pi\sqrt{12/5}}$, whereas the other factors does not vanish. Hence, we have
\begin{equation}\label{d17}
t^2-28t+61=0,
\end{equation}
where $t:=h^2+\displaystyle\frac{1}{h^2}.$\\
On solving the equation \eqref{d17} for $h$ and $t>0$, we deduce
\begin{equation}\label{d18}
h^2+\displaystyle\frac{1}{h^2}=14+3\sqrt{15}.
\end{equation}
On solving the equation \eqref{d18} for $h$ and $0<h<1$, we arrive at \eqref{d13} and \eqref{d14}.
\end{proof}


\begin{thebibliography}{}
\bibitem{CA3}
C. Adiga, Taekyun Kim,  M. S. Mahadeva Naika and H. S. Madhusudhan, On Ramanujan's cubic continued fraction and explicit evaluations of theta-functions, Indian J. pure appl. math., 35 (9), (2004), 1047--1062.
\bibitem{NDB}
N. D. Baruah and N. Saikia, Two parameters for Ramanujan's theta-functions and their explicit values, Rocky Mountain J. Math., 37 (6), (2007), 1747--1790.
\bibitem{BCB1}
 B. C. Berndt, Ramanujan's Notebooks, Part III, Springer-Verlag, New York, 1991.
\bibitem{BCB2}
 B. C. Berndt, Ramanujan's Notebooks, Part IV, Springer-Verlag, New York, 1994.
\bibitem{BCBLCZ1}
 B. C. Berndt and L. -C. Zhang, Ramanujan's identities for eta-functions, Math. Ann.,  292 (1), (1992), 561-573.
  \bibitem{BCBLCZ2}
 B. C. Berndt and L. -C. Zhang, A new class of theta-function identities originating in Ramanujan's notebooks, J. number theory, 48 (2), (1994), 224-242.
 \bibitem{BAM1}
 S. Bhargava, C. Adiga  and M. S. Mahadeva Naika, A new class of modular equations in Ramanujan's alternative theory of elliptic function of signature 4 and some new $P$--$Q$ eta--function identities, Indian J. Math., 45 (1), (2003), 23--39.
\bibitem{MSM5}
M. S. Mahadeva Naika, B. N. Dharmendra and S. Chandankumar, New modular relations for Ramanujan's parameter $\mu(q)$, Int. J. Pure Appl. Math., 74 (4), (2012), 413--435.
\bibitem{MSMMCMKSB}
 M. S. Mahadeva Naika, M. C. Maheshkumar and K. Sushan Bairy, On some remarkable product of theta-function, Aust. J. Math. Anal. Appl., 5 (1), (2008), Art. 13, 1--15.
\bibitem{MSMCKSHM}
M. S. Mahadeva Naika, S. Chandankumar and M. Harish, On some new P-Q mixed modular equations, Ann. Univ. Ferrara Sez. VII (N.S.), 59 (2), (2013), 353--374.
\bibitem{MSMCKSBH}
M. S. Mahadeva Naika, S. Chandankumar and B. Hemanthkumar, Modular relations for Ramanujan's remarkable product of theta functions, Adv. Stud. Contemp. Math., 23 (3), (2013), 431--449.
\bibitem{MSMCKSBH1}
M. S. Mahadeva Naika, S. Chandankumar and B. Hemanthkumar, On some new Modular relations for a remarkable product of theta--functions, Tbil. Math. J., 7 (1), (2014), 55--68.
\bibitem{SR2}
 S. Ramanujan, Notebooks (2 volumes), Tata Institute of Fundamental Research, Bombay, 1957.
\bibitem{SR1}
 S. Ramanujan, The lost notebook and other unpublished papers, Narosa, New Delhi, 1988.
\end{thebibliography}
\end{document}